\documentclass{amsart}
\usepackage{amsfonts}
\usepackage{amsmath}
\usepackage{graphicx}
\usepackage[english]{babel}

\newtheorem{theorem}{Theorem}

\newtheorem{corollary}[theorem]{Corollary}

\newtheorem{definition}[theorem]{Definition}
\newtheorem{example}[theorem]{Example}

\newtheorem{lemma}[theorem]{Lemma}

\newtheorem{proposition}[theorem]{Proposition}
\newtheorem{remark}[theorem]{Remark}

\newcommand{\Frac}[2]{\displaystyle\frac{#1}{#2}}
\newcommand{\Sum}{\displaystyle\sum}
\newcommand{\Lim}{\displaystyle\lim}

\begin{document}
\title[On the real projections of zeros of APF]{On the real projections of zeros of almost periodic functions} 

\author{J.M. Sepulcre}
\address{Department of Mathematics\\ University of
Alicante, 03080-Alicante\\
Spain} \email{JM.Sepulcre@ua.es}

\author{T. Vidal}
\address{University of
Alicante, 03080-Alicante\\
Spain} \email{tmvg@alu.ua.es}





\begin{abstract}
This paper deals with the set of the real projections of the zeros of an arbitrary almost periodic function defined in a vertical strip $U$. It provides practical results in order to determine whether a real number belongs to the closure of such a set. Its main result shows that, in the case that the Fourier exponents $\{\lambda_1,\lambda_2,\lambda_3,\ldots\}$  of an almost periodic function are linearly independent over the rational numbers, such a set has no isolated points in $U$.

\end{abstract}

\maketitle

\section{Introduction}

The theory of almost periodic functions, which was created and developed in its main features by H. Bohr during the $1920$'s, opened a way to study a wide class of trigonometric series of the general type and even exponential series. 
This theory shortly acquired numerous applications to various areas of mathematics, from harmonic analysis to differential equations. In the case of the functions that are defined on the real numbers, the notion of almost periodicity is a generalization of purely periodic functions and, in fact, as in classical Fourier analysis, any almost periodic function is associated with a Fourier series with real frequencies. 

Let us briefly recall some notions concerning the theory of the almost periodic functions of a complex variable, which was theorized in \cite{BohrAnalytic} (see also \cite{Besi,Bohr,Corduneanu1,Jessen}).
A function $f(s)$, $s=\sigma+it$, analytic in a vertical strip $U=\{s=\sigma+it\in\mathbb{C}:\alpha<\sigma<\beta\}$ ($-\infty\leq\alpha<\beta\leq\infty$), is called almost periodic in $U$ if to any
 $\varepsilon>0$ there exists a number $l=l(\varepsilon)$ such that each interval $t_0<t<t_0+l$ of length $l$ contains a number $\tau$ satisfying
$$|f(s+i\tau)-f(s)|\leq \varepsilon\ \forall s\in U.$$
We will denote as $AP(U,\mathbb{C})$ the space of almost-periodic functions in a vertical strip $U$.
It is known that any almost periodic function in $AP(U,\mathbb{C})$ is determined by an exponential series of the form $\sum_{n\geq 1} a_ne^{\lambda_n s}$ with complex coefficients $a_n$ and real exponents $\lambda_n$, called Fourier exponents of $f$. This associated series is called the Dirichlet series of the given analytic almost periodic function (see \cite[p.147]{Besi}, \cite[p.77]{Corduneanu1} or \cite[p.312]{Jessen}). 

Moreover, the set of almost periodic functions in a vertical strip $U$ coincides with the set of the functions which can be approximated uniformly in every reduced strip by exponential polynomials of the form
\begin{equation}\label{exppoly}
a_1e^{\lambda_1 s}+\ldots+a_ne^{\lambda_{n}s},\
a_j\in\mathbb{C},\ n\geq 2
\end{equation}
where $\{\lambda_1,\lambda_2,\ldots,\lambda_{n}\}$ is an ordered set of real numbers (see for example \cite[Theorem 3.18]{Corduneanu1}). In fact, it is convenient to recall that, even in the case that the sequence of the partial sums of its Dirichlet series does not converge uniformly, there exists a sequence of exponential polynomials, called Bochner-Fej\'{e}r polynomials, of the type $P_k(s)=\sum_{j\geq 1}p_{j,k}a_je^{\lambda_js}$ where for each $k$ only a finite number of the factors $p_{j,k}$ differ from zero, which converges uniformly to $f$ in every reduced strip in $U$  and converges formally to the Dirichlet series on $U$ \cite[Polynomial approximation theorem, pgs. 50,148]{Besi}.

The study of the zeros of the class of exponential polynomials of type (\ref{exppoly}) has become a topic
of increasing interest, see for example \cite{Avellar,Borwein,Farag,Gaps,Moreno,PT,JMT,SV3,SV,SV1}. 
In this paper, we will study certain properties on the zeros of an almost periodic function $f(s)$
in its vertical strip of almost periodicity $U=\{s=\sigma+it:\alpha<\sigma<\beta\}$. Specifically, consider the values $a_{f}$ and $%
b_{f}$ defined as
\begin{equation}\label{an}
a_{f}:= \inf \left\{ \operatorname{Re}s:f(s)=0,\ s\in U\right\}
\end{equation}
and
\begin{equation}\label{bn}
\ b_{f}:= \sup \left\{ \operatorname{Re}s:f(s)=0,\ s\in U\right\}.
\end{equation}
In general, if $f$ has at least one zero in $U$ it is satisfied $-\infty\leq a_f\leq b_f\leq\infty$ (it also depends on $U$). 
Given such a function $f(s)$, if $a_f$ and $b_f$ are real numbers, the bounds $a_f$ and $b_f$ allow us to define an interval $I_{f}:= \left[ a_{f},b_{f}%
\right]$  which contains the closure of
the set of the real parts of the zeros of $f(s)$ in $U$.
If either $a_f=\infty$ or $b_f=\infty$, the interval $I_f$ is of the form $(-\infty,b_f]$, $[a_f,\infty)$ or $(-\infty,\infty)$.
In this paper, we will focus our attention on the set
\begin{equation}\label{1.1}
R_{f}:= \overline{\left\{ \operatorname{Re}s:f(s)=0,\ s\in U\right\}}\cap (\alpha,\beta).
\end{equation}%

In this respect, the density properties of the zeros of several groups of exponential polynomials have also become a topic of increasing interest.
In particular, the topological properties of the set $R_{\zeta_n}=\overline{\left\{
\operatorname{Re}s:\zeta_n(s)=0\right\}}$ associated with the partial sums $\zeta_n(s)=1+2^{-s}+\ldots+n^{-s}$, $n\geq 2$, of the Riemann zeta function has been studied from different approaches. For example, an auxiliary function
associated with $\zeta_n(s)$ was used in \cite[Theorem 9]{JMT} in order to
establish conditions to
decide whether a real number is in the set $R_{\zeta_n}$. This auxiliary function,
which is called in \cite[p. 163]{Spira2} the ``companion function"\ of $\zeta_n$,
can also be adapted from a known result of C.E. Avellar and J.K. Hale
\cite[Theorem 3.1]{Avellar} in order to obtain analytical criterions about $R_P$ in the more general case of
exponential polynomials $P(s)$ of type (\ref{exppoly}). 

In this paper, by analogy to the case of exponential polynomials,
we first introduce an auxiliary function, of countably many real variables, which is associated with a prefixed almost periodic function $f(s)$ in a vertical strip $U$ (see section \ref{sectionaux}).
Secondly, this auxiliary function leads us to a practical characterization of the points of the set $R_{f}$
associated with $f(s)$ (see Theorem \ref{th5}). See also Theorem \ref{lemma3} which provides another characterization of the points in $R_f$ and extends other results such as \cite[Lemma 3]{RACSAM} or \cite[Lemma 4]{Sep}.
Thirdly, we study the closure set of the real parts of the zeros of almost periodic functions $f(s)$ whose Fourier exponents $\{\lambda_1,\lambda_2,\ldots,\lambda_k,\ldots\}$
are linearly independent over the rational numbers (see section \ref{sectionqli}). Under these hypothesis,
this study provides a new pointwise characterization of the set $R_f$ in terms of the inequalities (\ref{PGp})
(see Theorem \ref{point}), which facilitates the obtaining of Proposition \ref{lem2}, about the boundary points of $R_f$, and
corollaries \ref{cor8} and \ref{cor9} about some extra conditions under which we can state that $R_f\neq\emptyset$
(concerning this topic, see also Example \ref{exampler}).
Finally, also under $\mathbb{Q}$-linear independence of $\{\lambda_1,\lambda_2,\ldots\lambda_k,\ldots\}$, with $k>2$,
we prove that the set of the real projections of the zeros of $f(s)$
has no isolated point in $U$ (see Theorem \ref{theorem}), which generalizes \cite[Theorem 7]{Gaps}.

\section{An auxiliary function associated with an almost periodic function}\label{sectionaux}

Let $\mathcal{S}_{\Lambda}$ denote the class consisting of exponential sums of the form
 \begin{equation*}\label{eqqnew}
\sum_{j\geq 1}a_je^{\lambda_jp},\ a_j\in\mathbb{C},\ \lambda_j\in\Lambda,
\end{equation*}
where $\Lambda=\{\lambda_1,\lambda_2,\ldots,\lambda_j,\ldots\}$ is an arbitrary countable set of distinct real numbers (not necessarily unbounded), which are called a set of exponents or frequencies.

Also, let
$G_{\Lambda}=\{g_1, g_2,\ldots, g_k,\ldots\}$ be a basis of the
vector space over the rationals generated by a set $\Lambda$ of exponents, 
which implies that $G_{\Lambda}$ is linearly independent over the rational numbers and each $\lambda_j$ is expressible as a finite linear combination of terms of $G_{\Lambda}$, say
$$\lambda_j=\sum_{k=1}^{q_j}r_{j,k}g_k,\ \mbox{for some }r_{j,k}\in\mathbb{Q}.$$
By abuse of notation, we will say that $G_{\Lambda}$ is a basis for $\Lambda$. Moreover, we will say that $G_{\Lambda}$ is an integral basis for $\Lambda$ when $r_{j,k}\in\mathbb{Z}$ for any $j,k$. Finally, we will say that $G_{\Lambda}$ is the \textit{natural basis} for $\Lambda$, and we will denote it as $G_{\Lambda}^*$, when it is constituted by elements in $\Lambda$ as follows. Firstly if $\lambda_1\neq 0$ then $g_1:=\lambda_1\in G_{\Lambda}^*$. Secondly, if $\{\lambda_1,\lambda_2\}$ are $\mathbb{Q}$-rationally independent, then $g_2:=\lambda_2\in G_{\Lambda}^*$. Otherwise, if $\{\lambda_1,\lambda_3\}$ are $\mathbb{Q}$-rationally independent, then $g_2:=\lambda_3\in G_{\Lambda}^*$, and so on.

When the formal series in $\mathcal{S}_{\Lambda}$ are handled as exponential sums of a complex variable on which
we fix a summation procedure, we will introduce an auxiliary function which will be an important tool in this paper.
To do this, we first consider the definition of the classes $\mathcal{D}_{\Lambda}$ of almost periodic functions in the following terms.

\begin{definition}\label{DF}
Let $\Lambda=\{\lambda_1,\lambda_2,\ldots,\lambda_j,\ldots\}$ be an arbitrary countable set of distinct real numbers. We will say that a function $f:U\subset\mathbb{C}\to\mathbb{C}$ 
is in the class $\mathcal{D}_{\Lambda}$ if it is an almost periodic function in $AP(U,\mathbb{C})$ 
whose associated Dirichlet series 
is of the form
 \begin{equation}\label{eqqo}
\sum_{j\geq 1}a_je^{\lambda_js},\ a_j\in\mathbb{C},\ \lambda_j\in\Lambda,
\end{equation}
where $U$ is a strip of the type $\{s\in\mathbb{C}: \alpha<\operatorname{Re}s<\beta\}$, with $-\infty\leq\alpha<\beta\leq\infty$.
\end{definition}

Now, with respect to our particular case of almost periodic functions with the Bochner-Fej\'{e}r summation method (see, in this regard, \cite[Chapter 1, Section 9]{Besi}),  to every almost periodic function $f\in \mathcal{D}_{\Lambda}$ we can associate an auxiliary function $F_f$ of countably many real variables as follows (see \cite{SV2}).
For this, let $2\pi \mathbb{Z}^{m}=\{(c_1,c_2,\ldots,c_m)\in\mathbb{R}^{m}:c_k=2\pi n_k, \mbox{with }n_k\in\mathbb{Z},\ k=1,2,\ldots,m\}$.

\begin{definition}\label{auxuliaryfunc}
Given $\Lambda=\{\lambda_1,\lambda_2,\ldots,\lambda_j,\ldots\}$ a set of exponents, let $f(s)\in\mathcal{D}_{\Lambda}$ be an almost periodic function in $\{s\in\mathbb{C}:\alpha<\operatorname{Re}s<\beta\}$, $-\infty\leq\alpha<\beta\leq\infty$, whose Dirichlet series is given by $\sum_{j\geq 1}a_je^{\lambda_js}$.
For each $j\geq 1$ let $\mathbf{r}_j$ be the vector of rational components satisfying the equality $\lambda_j=<\mathbf{r}_j,\mathbf{g}>=\sum_{k=1}^{q_j}r_{j,k}g_k$, where
$\mathbf{g}:=(g_1,\ldots,g_k,\ldots)$ is the vector of the elements of the natural basis $G_{\Lambda}^*$ for $\Lambda$.
Then we define the auxiliary function  $F_f: (\alpha,\beta)
\times
[0,2\pi)
^{\sharp G_{\Lambda}^*}\times \prod_{j\geq 1}2\pi\mathbb{Z}^{\sharp G_{\Lambda}^*}\rightarrow
\mathbb{C}$
 associated with $f$, relative to the basis $G_{\Lambda}^*$, as
\begin{equation}\label{2.4}
F_{f}(\sigma,\mathbf{x},\mathbf{p}_1,\mathbf{p}_2,\ldots):=\sum_{j\geq1}a_j e^{\lambda_j\sigma
}e^{<\mathbf{r}_j,\mathbf{x}+\mathbf{p}_j>i}\text{, }
\end{equation}%
where $\sigma \in
(\alpha,\beta)
\text{, }\mathbf{x}\in%
[0,2\pi)^{\sharp G_{\Lambda}^*},\ \mathbf{p}_j\in2\pi\mathbb{Z}^{\sharp G_{\Lambda}^*}$ and series (\ref{2.4}) is summed by Bochner-Fej\'{e}r procedure, applied at $t=0$ to the sum
$\sum_{j\geq1}a_j e^{<\mathbf{r}_j,\mathbf{x}+\mathbf{p}_j>i}e^{\lambda_js}$.
\end{definition}

We first note that, if $\sum_{j\geq1}a_j e^{\lambda_js}$ is the Dirichlet series of $f\in AP(U,\mathbb{C})$,
for every choice of $\mathbf{x}\in\mathbb{R}^{\sharp G_{\Lambda}}$ and $\mathbf{p}_j\in 2\pi\mathbb{Z}^{\sharp G_{\Lambda}}$, $j=1,2,\ldots$, the sum
$\sum_{j\geq1}a_j e^{<\mathbf{r}_j,\mathbf{x}+\mathbf{p}_j>i}e^{\lambda_js}$ represents the Dirichlet series of an
almost periodic function which is connected with $f$ through an equivalence relation (see \cite[Lemma 3]{SV}). 

We next introduce the following notation which will be used to show the direct relation between an almost periodic function and the auxiliary function associated with it.

\begin{definition}\label{image}
Given $\Lambda=\{\lambda_1,\lambda_2,\ldots,\lambda_j,\ldots\}$ a set of exponents, let $f(s)\in \mathcal{D}_{\Lambda}$ be an almost periodic function in an open vertical strip $U$, and $\sigma_0=\operatorname{Re}s_0$ with $s_0\in U$. We define $\operatorname{Img}\left(F_f(\sigma_0,\mathbf{x},\mathbf{p}_1,\mathbf{p}_2,\ldots)\right)$ to be the set of values in the complex plane taken on by the auxiliary function $F_f(\sigma,\mathbf{x},\mathbf{p}_1,\mathbf{p}_2,\ldots)$ when $\sigma=\sigma_0$; that is
$\operatorname{Img}\left(F_f(\sigma_0,\mathbf{x},\mathbf{p}_1,\mathbf{p}_2,\ldots)\right)=\{s\in\mathbb{C}:\exists \mathbf{x}\in[0,2\pi)^{\sharp G_{\Lambda}^*}\ \mbox{and }\mathbf{p}_j\in 2\pi\mathbb{Z}^{\sharp G_{\Lambda}^*}\mbox{ such that}\ s=F_f(\sigma_0,\mathbf{x},\mathbf{p}_1,\mathbf{p}_2,\ldots)\}.$
\end{definition}

The notation $\operatorname{Img}\left(F_f(\sigma_0,\mathbf{x},\mathbf{p}_1,\mathbf{p}_2,\ldots)\right)$ is well-posed because this set is independent of the basis $G_{\Lambda}$ (see \cite{SV2}). 

Also, given a function $f(s)$, take the notation $$\operatorname{Img}\left(f(\sigma_0+it)\right)=\{s\in\mathbb{C}:\exists t\in\mathbb{R}\mbox{ such that }s=f(\sigma_0+it)\}.$$



On the other hand, by analogy with Bohr's theory for Dirichlet series, we established in \cite{SV}
an equi\-valence relation $\shortstack{$_{{\fontsize{6}{7}\selectfont *}}$\\$\sim$}$ on the classes $\mathcal{S}_{\Lambda}$ and other
general classes of functions such as those in $\mathcal{D}_{\Lambda}$. For example, concerning the classes $\mathcal{D}_{\Lambda}$, it was proved in
\cite[Proposition 2, ii)]{SV2} that
$\operatorname{Img}\left(F_f(\sigma_0,\mathbf{x},\mathbf{p}_1,\mathbf{p}_2,\ldots)\right)=\bigcup_{f_k\shortstack{$_{{\fontsize{6}{7}\selectfont *}}$\\$\sim$} f}\operatorname{Img}\left(f_k(\sigma_0+it)\right)$. Moreover, it was proved in
\cite[Theorem 1]{SV2} that if $E$ is an open set of real numbers included in $(\alpha,\beta)$,
then $$\bigcup_{\sigma\in E}\operatorname{Img}\left(f_1(\sigma+it)\right)=\bigcup_{\sigma\in E}\operatorname{Img}\left(f_2(\sigma+it)\right),$$
where $f_1,f_2\in \mathcal{D}_{\Lambda}$ are two equivalent almost periodic functions in $U$.
That is, the functions $f_1$ and $f_2$ take the same set of values on the region $\{s=\sigma+it\in\mathbb{C}:\sigma\in E\}$.

\section{The closure set of the real projections of the zeros in the general case}

This section is mainly devoted to show two point-wise characterizations of the sets $R_f$ defined in (\ref{1.1}) associated with an almost periodic function $f(s)$.

We first extend and improve some results in the style of \cite[Lemma 3]{RACSAM} or \cite[Lemma 4]{Sep}.

\begin{theorem}\label{lemma3}
Let $f(s)$ be an almost periodic function in an open vertical strip $U=\{\sigma+it\in\mathbb{C}:\alpha<\sigma<\beta\}$, and
$\sigma _{0}\in\left(\alpha,\beta\right)$.  Then $\sigma_{0}\in R_{f}$ if
and only if there exists a sequence
$\{t_{j}\}_{j=1,2,\ldots}$ of real numbers such that
\begin{equation}\label{unoooo}
\lim_{j\rightarrow \infty }f(\sigma _{0}+it_{j})=0\text{.}
\end{equation}%
\end{theorem}
\begin{proof}
Note that $f(s)$, and its derivatives, are bounded on every reduced strip of $U$ \cite[pp.142-144]{Besi}. 
Suppose first the existence of $\{t_{j}\}_{j=1,2,\ldots}\subset \mathbb{R}$ satisfying \eqref{unoooo}.  In order to apply \cite[Lemma, p.73]{Moreno}, we next prove that there exist positive numbers $\delta $ and $l$ such that on any
segment of length $l$ of the line $x=\sigma _{0}$ there is a point
$\sigma _{0}+iM$ such that $\left\vert f(\sigma
_{0}+iM)\right\vert \geq \delta $. Indeed, let $t_0$ be a real number such that
$|f(\sigma_0+it_0)|>0$ and take
$\delta=\Frac{|f(\sigma_0+it_0)|}{2}$. Then, since $f(s)$ is an
almost-periodic function in $U$, there exists a positive real number $l=l(\delta)$%
\ such that every interval of length $l$ on the imaginary axis
contains at
least one translation number $iT$, associated with $\delta $, satisfying $%
\left\vert f(s+iT)-f(s)\right\vert \leq \delta $ for all $s\in
U
$. Thus, by taking $s=\sigma _{0}+it_0$, we have $\left\vert
f(\sigma _{0}+i(t_0+T))-f(\sigma_{0}+it_0)\right\vert \leq \delta
$ and, according to the choice of $\delta$, it follows that
$\left\vert f(\sigma _{0}+i(t_0+T))\right\vert \geq \delta $.
Therefore, $\left\vert f(\sigma _{0}+iM)\right\vert \geq \delta $,
with $M:=t_0+T$. Consequently, the function $f(s)$ has the properties needed to
apply \cite[Lemma, p.73]{Moreno} and thus $f(s)$ has zeros
in any strip $$S_{\epsilon }:=\left\{ s\in\mathbb{C}:\sigma _{0}-\epsilon
<\operatorname{Re}s<\sigma _{0}+\epsilon \right\},$$ for any arbitrary
$\epsilon >0$, which proves that $\sigma _{0}\in R_{f}$.

Conversely suppose that $\sigma_{0}\in R_{f}$. This means that there exists a sequence $\{s_j\}_{j\geq 1}\subset U$, with $s_j=\sigma_j+it_j$, such that $f(s_j)=0$ and $\lim_{j\to\infty}\sigma_j=\sigma_0$. Consider the sequence of functions given by
$f_j(s):=f(s+t_j)$, $s\in U$, $j=1,2,\ldots$ which is uniformly bounded on every strip $S_{\epsilon}\subset U$.
By Montel's theorem \cite[Section 5.1.10]{Ash}, $\{f_j(s)\}_{j\geq 1}$ has a subsequence, which we denote again by $\{f_{j}(s)\}$, converging uniformly on compact subsets of $S_{\epsilon}$ to a function $h(s)$ analytic in $S_{\epsilon}$.
In this way, it is clear that $$h(\sigma_0)=\lim_{j\to\infty}f_j(\sigma_0)=\lim_{j\to\infty}f(\sigma_0+it_j)=0.$$
Indeed, note that $h(\sigma_0)=0$ in virtue from
$f_j(\sigma_j)=0$, $j=1,2,\ldots$, and $\lim_{j\to\infty}\sigma_j=\sigma_0$. 
\end{proof}

We next give a characterization of the sets $%
R_{f}$ by means of an \textit{ad hoc }version of \cite[Theorem
3.1]{Avellar}, which is obtained through the auxiliary function $F_f$ analysed in the previous section (see (\ref{2.4})).

\begin{theorem}\label{th5}
Let $f(s)$ be an almost periodic function in a vertical strip $U=\{s=\sigma+it:\alpha<\sigma<\beta\}$.
Consider $\sigma\in(\alpha,\beta)$.
Then $\sigma \in R_{f}$ if and only if there exist
some vectors $\mathbf{x}\in[0,2\pi)^{\sharp G_{\Lambda}^*}$ and $\mathbf{p}_j\in 2\pi\mathbb{Z}^{\sharp G_{\Lambda}^*}$ such that $F_{f}(\sigma ,\mathbf{x},\mathbf{p}_1,\mathbf{p}_2,\ldots)=0$, where $\Lambda$ is the
set of Fourier exponents of $f(s)$.
\end{theorem}
\begin{proof}
Let $\sigma_0\in R_f$. Then there exists a sequence
$\{s_{j}\}_{j=1,2,\ldots}\subset U$, with $s_j=\sigma _{j}+it_{j}$, of zeros of $f(s)$
such that $\sigma_0
=\lim_{j\rightarrow \infty }\sigma _{j}$. Consider the sequence of functions $\{f_j(s)\}_{j\geq 1}$
defined as $f_j(s):=f(s+it_j)$, which are analytic in $U$. Then it is clear that each $f_j(s)$
is equivalent to $f(s)$ (see \cite[Lemma 1]{SV}), that is $f_j\ \shortstack{$_{{\fontsize{6}{7}\selectfont *}}$\\$\sim$}\  f$ where $\shortstack{$_{{\fontsize{6}{7}\selectfont *}}$\\$\sim$}$ is the equivalence relation considered in \cite{SV,SV2}.
Now, by taking into account \cite[Propositions 3 and 4]{SV}, we can extract a subsequence of $\{f_j(s)\}_{j\geq 1}$
which converges uniformly on every reduced strip of $U$ to a function $h(s)$ in the same equivalence class as $f$.
Moreover, we have $h(\sigma_0)=0$. Indeed $h(\sigma_0)=\lim_{j\to \infty}f_j(\sigma_0)=0$.
Otherwise, there would exist $D(\sigma_0,\varepsilon)$ such that $h(s)\neq 0$ $\forall s\in \overline{D}(\sigma_0,\varepsilon)$
and, by Hurwitz's theorem \cite[Section 5.1.3]{Ash}, there would exist $j_0\in\mathbb{N}$ such that
$f_j(s)\neq 0$ $\forall s\in \overline{D}(\sigma_0,\varepsilon)$ and each $j\geq j_0$,
which is a contradiction because $f_j(\sigma_j)=0$. Consequently, as $h(\sigma_0)=0$ and $\operatorname{Img}\left(F_f(\sigma_0,\mathbf{x},\mathbf{p}_1,\mathbf{p}_2,\ldots)\right)=\bigcup_{f_k \shortstack{$_{{\fontsize{6}{7}\selectfont *}}$\\$\sim$}f}\operatorname{Img}\left(f_k(\sigma_0+it)\right)$ \cite[Proposition 2, ii)]{SV2}, we have that there exist $\mathbf{x}\in[0,2\pi)^{\sharp G_{\Lambda}^*}$ and $\mathbf{p}_j\in 2\pi\mathbb{Z}^{\sharp G_{\Lambda}^*}$ such that $0=F_f(\sigma_0,\mathbf{x},\mathbf{p}_1,\mathbf{p}_2,\ldots)$.

Conversely, suppose that
$F_{f}(\sigma ,\mathbf{x},\mathbf{p}_1,\mathbf{p}_2,\ldots)=0$ for some real number $%
\sigma\in (\alpha,\beta)$ and some vectors $\mathbf{x}=(x_{1}, x_{2},\ldots,x_k,\ldots)\in[0,2\pi)^{\sharp G_{\Lambda}^*}$ and $\mathbf{p}_j\in 2\pi\mathbb{Z}^{\sharp G_{\Lambda}^*}$.
Again by \cite[Proposition 2, ii)]{SV2}, we have that
$\operatorname{Img}\left(F_f(\sigma,\mathbf{x})\right)=\bigcup_{f_k\shortstack{$_{{\fontsize{6}{7}\selectfont *}}$\\$\sim$} f}\operatorname{Img}\left(f_k(\sigma+it)\right)$. Hence there exists $f_k\ \shortstack{$_{{\fontsize{6}{7}\selectfont *}}$\\$\sim$}\ f$ such that $f_k(\sigma+it)=0$ for some real number $t$. Furthermore, by \cite[Theorem 1]{SV2}, for any $\varepsilon>0$ sufficiently small it is accomplished that the functions $f_k$ and $f$ take the same set of values on the region $\{s\in\mathbb{C}:\operatorname{Re}s\in (\sigma-\varepsilon,\sigma+\varepsilon)\}$. This means that $\sigma\in R_f$.
\end{proof}

\section{The closure set of the real projections of the zeros under $\mathbb{Q}$-linear independence of the Fourier exponents}\label{sectionqli}

We first recall that if the Fourier exponents of an almost-periodic function $f$ in
a vertical strip $\{s=\sigma+it:\alpha<\sigma<\beta\}$ are linearly independent over the rational numbers, then the Dirichlet series expansion of $f$ converges to $f$ itself and, in fact, it converges absolutely in $(\alpha,\beta)$ (\cite[Theorem 3.6]{Corduneanu} and \cite[p. 154]{Besi}). Moreover, in this case it is obvious that the set of Fourier exponents has an integral basis.

We next prove the following characterization of the points in the set $R_f$ associated with an almost periodic function $f$
whose Fourier exponents are $\mathbb{Q}$-linearly independent, that is, when its exponents are linearly independent over
the rational numbers.

\begin{theorem}\label{point}
Let $f(s)$ be an almost periodic function in a vertical strip $U=\{s=\sigma+it:\alpha<\sigma<\beta\}$ whose Dirichlet series is given by $\sum_{n\geq 1} a_ne^{\lambda_n s}$
with $\{\lambda_1,\lambda_2,\ldots,\lambda_k,\ldots\}$ $\mathbb{Q}$-linearly independent and $k>2$. Let $\sigma_0\in(\alpha,\beta)$. Then $\sigma_0 \in R_{f}$ if and only if
\begin{equation}\label{PGp}
\left\vert a_{j}\right\vert e^{\sigma_0 \lambda_{j}} \leq \sum_{i\geq1\text{, }%
i\neq j}\left\vert a_{i}\right\vert e^{\sigma_0 \lambda_{i}}\ %
\left( j=1,2,\ldots,k,\ldots\right).
\end{equation}%
\end{theorem}
\begin{proof}
Without loss of generality, take $G_{\Lambda}^*=\{\lambda_1,\lambda_2,\ldots\}$ as the basis of the
vector space over the rationals generated by the set of Fourier exponents of $f$.
Suppose that $\sigma_0\in R_f$, then by Theorem \ref{th5} there exist some vectors $\mathbf{x}\in[0,2\pi)^{\sharp G_{\Lambda}^*}$ and $\mathbf{p}_j\in 2\pi\mathbb{Z}^{\sharp G_{\Lambda}^*}$ such that $F_{f}(\sigma ,\mathbf{x},\mathbf{p}_1,\mathbf{p}_2,\ldots)=0$ or, equivalently,
$\Sum_{n\geq1}a_ne^{\lambda_n\sigma_0}e^{x_ni}=0$ (by taking $\mathbf{g}=(\lambda_1,\lambda_2,\ldots)$ and hence $r_{n,k}=0$ if $k\neq n$ and $r_{n,n}=1$). Therefore,
$$a_je^{\lambda_j\sigma_0}e^{x_ji}=-\Sum_{k\geq1,k\neq j}a_ke^{\sigma_0 \lambda_k}e^{x_ki},\ j=1,2,\ldots$$
and, by taking the modulus, we get
$$|a_j|e^{\lambda_j\sigma_0}\leq\sum_{k\geq1,k\neq j}|a_k|e^{\sigma_0 \lambda_k},\ j=1,2,\ldots.$$

Conversely, suppose that the positive real numbers
$|a_j|e^{\sigma_0 \lambda_j}$, $j=1,2,\ldots$, satisfy inequalities
(\ref{PGp}). We recall that by \cite[Theorem 3.6]{Corduneanu} or \cite[p.154]{Besi} it is accomplished $\sum_{j\geq 1}|a_j|e^{\sigma_0 \lambda_j}<\infty$. Thus, given $\varepsilon>0$ there exists $n_0\in \mathbb{N}$ such that $\sum_{j\geq n_0}|a_j|e^{\sigma_0 \lambda_j}<\varepsilon$. Hence, for $\varepsilon>0$ sufficiently small
we can index the terms in decreasing order so that $m_1$ is
such that $|a_{m_1}|e^{\sigma_0 \lambda_{m_1}}:=\max\{|a_k|e^{\sigma_0
\lambda_k}:k=1,2,\ldots,n_0-1\}$, $m_2$ such that $|a_{m_2}|e^{\sigma_0
\lambda_{m_2}}:=\max\{|a_k|e^{\sigma_0 \lambda_k}:k=1,2,\ldots,n_0-1, k\neq
m_1\},$ etc. Therefore, by taking $r:=\sum_{j\geq n_0}|a_j|e^{\sigma_0 \lambda_j}$, there is at least one $n_0$-sided polygon whose
sides have the lengths $|a_{m_j}|e^{\sigma_0 \lambda_{m_j}}$, $j=1,2,\ldots,n_0-1$ and $r$ \cite[p.71]{Moreno}. That means that
there exist real numbers $\theta_1,\theta_2,\ldots,\theta_{n_0}$
satisfying
$$\Sum_{k=1}^{n_0-1}|a_{k}|e^{\sigma_0 \lambda_{k}}e^{i\theta_k}+re^{i\theta_{n_0}}=0.$$
Consequently, by taking $\mathbf{x}\in \mathbb{R}^{\sharp G_{\Lambda}}$ the vector given by $x_{k}=\theta_k-\operatorname{Arg}(a_{k})$ for
$k=1,\ldots,n_0-1$, and $x_{k}=\theta_{n_0}-\operatorname{Arg}(a_{k})$ for each $k\geq n_0$, where $\operatorname{Arg}(a_k)$ denotes the
principal argument of $a_k$, we have
$$F_{f}(\sigma_0,\mathbf{x},\mathbf{0},\mathbf{0},\ldots)=\sum_{k\geq1}a_ke^{\lambda_k\sigma_0
}e^{x_ki}=0.$$ 
Hence, from Theorem \ref{th5}, $\sigma_0\in R_f$.
\end{proof}




From now on we will analyse some properties of the set $R_{f}$ defined by $$\overline{\left\{ \operatorname{Re}%
s:f(s)=0,\ s\in U\right\} }\cap (\alpha,\beta),$$ associated with an almost periodic function
$f(s)$ in a vertical strip $U=\{s=\sigma+it:\alpha<\sigma<\beta\}$ with rationally independent Fourier exponents. 
%
%
Given such a function $f(s)$, at
any boundary point of the set $R_{f}$, we next prove that, concerning inequalities (\ref{PGp}), the equality is attained in
only one of these inequalities.

\begin{proposition}\label{lem2}
Let $f(s)$ be an almost periodic function in a vertical strip $U=\{s=\sigma+it:\alpha<\sigma<\beta\}$ whose Fourier exponents $\{\lambda_1,\lambda_2,\ldots,\lambda_k,\ldots\}$, with $k>2$, are $\mathbb{Q}$-linearly independent. Let $\sigma_0\in (\alpha,\beta)$. If
$\sigma_{0}$ is a boundary point of $R_{f}$, then it satisfies all
the inequalities (\ref{PGp}) and only one of them is an
equality.
\end{proposition}
\begin{proof}
Let $f(s)$ be an almost periodic function in a vertical strip $U=\{s=\sigma+it:\alpha<\sigma<\beta\}$ whose Dirichlet series is given by
$\sum_{n\geq 1}a_ne^{\lambda_ns}$, with $\{\lambda_1,\lambda_2,\ldots,\lambda_k,\ldots\}$ $\mathbb{Q}$-linearly independent and $k>2$. As $R_{f}$ is closed in $(\alpha,\beta)$, the boundary of $R_{f}$ 
is a subset of $R_{f}$ itself. Then $\sigma _{0}\in R_{f}$ and, by Theorem \ref{point},
inequalities (\ref{PGp}) are
obviously satisfied for $\sigma_{0}$. Moreover,
if some of the inequalities (\ref{PGp}) is an equality, as any couple of
equalities are incompatible, the lemma follows. Otherwise we have
the following strict inequalities 
\begin{equation}\label{uno}
\left\vert a_{j}\right\vert e^{\sigma_0 \lambda_{j}}< \sum_{i\geq1\text{, }%
i\neq j}\left\vert a_{i}\right\vert e^{\sigma_0 \lambda_{i}} \text{, }%
\left( j=1,2,\ldots,k,\ldots\right).
\end{equation}%
Now, as $\sigma_0$ is a boundary point of $R_{f}=\overline{\left\{ \operatorname{Re}s:f(s)=0,\ s\in U\right\}}\cap (\alpha,\beta)$, there exists $\varepsilon>0$ such that either $R_f^c\supset (\sigma_0-\varepsilon,\sigma_0)$ or $R_f^c\supset (\sigma_0,\sigma_0+\varepsilon)$, where $R_f^c$ denotes $(\alpha,\beta)\setminus R_f$. Let
$\sigma_1$ be a point in $(\sigma_0-\varepsilon,\sigma_0+\varepsilon)\cap R_f^c$. In virtue of Theorem \ref{point}, it is plain that there exists a single $j_0\geq 1$ so that
\begin{equation}\label{dos}
\left\vert a_{j_0}\right\vert e^{\sigma_1 \lambda_{j_0}} > \sum_{i\geq1\text{, }%
i\neq j_0}\left\vert a_{i}\right\vert e^{\sigma_1 \lambda_{i}}.
\end{equation}
Thus, by continuity, we deduce from (\ref{uno}) and (\ref{dos}), for this $j_0$, that there exists $\sigma_2$ between $\sigma_0$ and $\sigma_1$ such that 
\begin{equation}\label{tres}
\left\vert a_{j_0}\right\vert e^{\sigma_2 \lambda_{j_0}}= \sum_{i\geq1\text{, }%
i\neq j_0}\left\vert a_{i}\right\vert e^{\sigma_2 \lambda_{i}}.
\end{equation}
Moreover, by \eqref{tres} it is clear that
\begin{equation*}\label{cuatro}
\left\vert a_{j}\right\vert e^{\sigma_2 \lambda_{j}}\leq \sum_{i\geq1\text{, }%
i\neq j}\left\vert a_{i}\right\vert e^{\sigma_2 \lambda_{i}} \text{, }%
\left( j=1,2,\ldots,k,\ldots\right).
\end{equation*}%
which implies, again by Theorem \ref{point}, that $\sigma_2\in R_f$. This is a contradiction and hence the result holds.
\end{proof}

We first prove that the set $R_f$ associated with an almost-periodic function in $U=\mathbb{C}$, whose Fourier exponents are linearly independent over the rational numbers, is not the empty set. 

\begin{corollary}\label{cor8}
Let $f(s)$ be an almost periodic function in $\mathbb{C}$ whose Fourier exponents
$\{\lambda_1,\lambda_2,\ldots,\lambda_k,\ldots\}$, with $k>2$, are $\mathbb{Q}$-linearly independent.
Then $R_{f}\neq\emptyset$. 
\end{corollary}
\begin{proof}
Let $f(s)$ be an almost periodic function in $\mathbb{C}$ whose Dirichlet series is given by $\sum_{n\geq 1} a_ne^{\lambda_n s}$. By reductio ad absurdum, suppose $R_{f}=\emptyset$. By Theorem \ref{point}, there exists $j_0\geq 1$ so that
\begin{equation}\label{doss}
\left\vert a_{j_0}\right\vert e^{\sigma \lambda_{j_0}} > \sum_{i\geq1\text{, }%
i\neq j_0}\left\vert a_{i}\right\vert e^{\sigma \lambda_{i}},\ \mbox{for all }\sigma\in\mathbb{R}.
\end{equation}
Otherwise (if the inequality is not true for all $\sigma\in\mathbb{R}$), by continuity, there would exist $\sigma_0\in\mathbb{R}$ such that $\left\vert a_{j_0}\right\vert e^{\sigma_0 \lambda_{j_0}} = \sum_{i\geq1\text{, }%
i\neq j_0}\left\vert a_{i}\right\vert e^{\sigma_0 \lambda_{i}}$ and hence, by Theorem \ref{point}, it is clear that $\sigma_0$ would be in $R_f$.

However, we next show that (\ref{doss}) is a contradiction. Indeed, take $k\neq j_0$. Since $\lambda_k\neq \lambda_{j_0}$, it is plain that there exists $\sigma_0\in\mathbb{R}$ such that $|a_k|e^{\sigma_0\lambda_k}=|a_{j_0}|e^{\sigma_0\lambda_{j_0}}$, which implies that
$$\sum_{i\geq1\text{, }%
i\neq j_0}\left\vert a_{i}\right\vert e^{\sigma_0 \lambda_{i}}>|a_k|e^{\sigma_0\lambda_k}=|a_{j_0}|e^{\sigma_0\lambda_{j_0}},$$
which contradicts (\ref{doss}). Now the result holds.
\end{proof}

It is worth noting that the condition $U=\mathbb{C}$ in Corollary \ref{cor8} is necessary.
That is, an almost periodic function $f(s)$ in a vertical strip $U$, with $U\neq \mathbb{C}$, such that its Fourier exponents
are $\mathbb{Q}$-linearly independent could satisfy $R_f=\emptyset$.
\begin{example}\label{exampler}
Let $\{1,\lambda_1,\lambda_2,\ldots,\lambda_n,\ldots,\rho_1,\rho_2,\ldots,\rho_n,\ldots\}$
be an ordered set of positive real numbers which are linearly
independent over the rational numbers and satisfy $\lambda_n>n^2$ (hence $\rho_n>n^2$) for each $n=1,2,\ldots$.
It is clear that such a set can be constructed from an arbitrary Hamel basis for $\mathbb{R}$ over $\mathbb{Q}$ in virtue of the density of rational numbers in $\mathbb{R}$. 

On the one hand, consider the exponential sum
$$S_1(s)=\sum_{n\geq 1}\frac{1}{\lambda_n}e^{\lambda_ns}.$$
Note that the exponents of $S_1$ are linearly independent over the rational numbers.
Moreover, $S_1(s)$ converges absolutely on $U_1=\{s=\sigma+it\in\mathbb{C}:\sigma< 0\}$.
Indeed, for any $\sigma< 0$, it is satisfied that
$$\sum_{n\geq 1}\frac{1}{\lambda_n}e^{\lambda_n\sigma}\leq\sum_{n\geq 1}\frac{1}{\lambda_n}\leq \sum_{n\geq 1}\frac{1 }{n^2}=\frac{ \pi^2}{ 6}$$
and the derivative of $S_1(s)$ accomplishes
$$S_1'(\sigma)=\sum_{n\geq 1}e^{\lambda_n\sigma},$$
which implies that  $S_1(s)$ diverges when $\sigma>0$.

On the other hand, consider the exponential sum
$$S_2(s)=\sum_{n\geq 1}\frac{1}{\rho_n}e^{-\rho_n(s+1)}=\sum_{n\geq 1}\frac{1}{\rho_n}e^{-\rho_n}e^{-\rho_ns}.$$
Note that the exponents of $S_2$ are linearly independent over the rational numbers.
Moreover, $S_2(s)$ converges absolutely on $U_2=\{s=\sigma+it\in\mathbb{C}:\sigma>-1\}$.
Indeed, for any $\sigma>-1$, it is satisfied that
$$\sum_{n\geq 1}\frac{1}{\rho_n}e^{-\rho_n}e^{-\rho_n\sigma}\leq\sum_{n\geq 1}\frac{1}{\rho_n}\leq \sum_{n\geq 1}\frac{1 }{n^2}=\frac{ \pi^2}{ 6}$$
and the derivative accomplishes
$$S_2'(\sigma)=\sum_{n\geq 1}-e^{-\rho_n(\sigma+1)},$$
which implies that  $S_2(s)$ diverges when $\sigma<-1$.

In this way, let $S(s):=S_1(s)+S_2(s)+\Frac{e\pi^2}{3}e^s$, $s\in\mathbb{C}$.
By above, it is accomplished that $S(s)$ is an exponential sum, whose exponents are linearly independent over the rationals,
which converges uniformly in
$U=U_1\cap U_2=\{s=\sigma+it\in\mathbb{C}:-1<\sigma<0\}$. In fact, $U$ is the largest open vertical strip of almost periodicity of $S(s)$.
Moreover, $$|S_1(s)+S_2(s)|\leq2\frac{ \pi^2}{ 6}=\frac{ \pi^2}{ 3}< \left|\frac{e\pi^2}{3}e^s\right|\ \forall s\in U.$$
Consequently, $S(s)$ has no zeros in $U$ and hence $R_S=\emptyset$.
\end{example}

In this respect, we next study the following more general result.
\begin{corollary}\label{cor9}
Let $f(s)$ be an almost periodic function in a vertical strip $U=\{s=\sigma+it:\alpha<\sigma<\beta\}$
whose Dirichlet series is given by $\sum_{n\geq 1} a_ne^{\lambda_n s}$ and $\{\lambda_1,\lambda_2,\ldots,\lambda_k,\ldots\}$, with $k>2$, are $\mathbb{Q}$-linearly independent.
Suppose that $U$ is the largest open vertical strip of almost periodicity of $f$ and some of the following conditions is satisfied:
\begin{itemize}
\item[a)] $\Lim_{\sigma\to\alpha^+}\sum_{i\geq1}|a_i|e^{\sigma\lambda_i}>2\sup\{|a_i|e^{\alpha\lambda_i}:i\geq 1\}$; 
\item[b)] $\Lim_{\sigma\to\beta^-}\sum_{i\geq1}|a_i|e^{\sigma\lambda_i}>2\sup\{|a_i|e^{\beta\lambda_i}:i\geq 1\}$. 
\end{itemize}
Then $R_{f}\neq\emptyset$.
\end{corollary}
\begin{proof}
Let $f(s)$ be an almost periodic function in $U$ whose Dirichlet series is given by $\sum_{n\geq 1} a_ne^{\lambda_n s}$ and $\{\lambda_1,\lambda_2,\ldots,\lambda_k,\ldots\}$, with $k>2$, are $\mathbb{Q}$-linearly independent.
By \cite[p. 154, first theorem]{Besi}, $\sum_{n\geq 1} a_ne^{\lambda_n s}$ is absolutely convergent in $U$ and hence $f(s)=\sum_{n\geq 1} a_ne^{\lambda_n s}$ for any $s\in U$.
Let $f_1(s):=\sum_{n\geq 1} |a_n|e^{\lambda_n s}$, which is clearly anaytic on $U$, then $U$ is the largest open vertical
strip of almost periodicity of $f_1$. Indeed, it is plain that $f_1$ is almost periodic on $U$ (it converges absolutely on $U$).
Moreover, if there was an open vertical strip $V\supset U$, with $U\neq V$ and $f_1$ almost periodic on $V$, then
$$\left|\sum_{n\geq 1} a_ne^{\lambda_n s}\right|\leq \sum_{n\geq 1} |a_n|e^{\lambda_n \sigma}=f_1(\sigma)<\infty\ \forall s=\sigma+it\in V$$
and hence $V$ would be an open vertical strip where $f$ converges absolutely, which is a contradiction.

Now, by reductio ad absurdum, suppose $R_{f}=\emptyset$. By Theorem \ref{point}, there exists $j_0\geq 1$ so that
\begin{equation}\label{doss0}
 |a_{j_0}| e^{\sigma \lambda_{j_0}} > \sum_{i\geq1\text{, }%
i\neq j_0}|a_{i}|e^{\sigma \lambda_{i}},\ \mbox{for all }\sigma\in(\alpha,\beta).
\end{equation}
Otherwise (if the inequality is not true for all $\sigma\in(\alpha,\beta)$), by continuity, there would exist
$\sigma_0\in(\alpha,\beta)$ such that $ |a_{j_0}|e^{\sigma_0 \lambda_{j_0}} = \sum_{i\geq1\text{, }%
i\neq j_0} |a_{i}| e^{\sigma_0 \lambda_{i}}$ and hence, by Theorem \ref{point}, it is clear that $\sigma_0$ would be in $R_f$.
Let $g(\sigma):=\sum_{i\geq1\text{, }%
i\neq j_0}|a_{i}| e^{\sigma \lambda_{i}}-|a_{j_0}|e^{\sigma \lambda_{j_0}},\ \sigma\in(\alpha,\beta).
$ We deduce from \eqref{doss0}  that
\begin{equation}\label{hut0}
g(\sigma)<0 \mbox{ for all }\sigma\in(\alpha,\beta).
\end{equation}
However, we next show that this is a contradiction.
Note first that, by taking into account \cite[p. 154, second theorem]{Besi}, $\alpha$ and $\beta$ are singular points of $f_1(s)$.
By hypothesis (condition a)), given $\varepsilon>0$,
there exists
$\sigma_1\in(\alpha,\alpha+\varepsilon)$ such that
\begin{equation*}\label{lki12}
\sum_{i\geq1}|a_i|e^{\sigma\lambda_i}>2\sup\{|a_i|e^{\alpha\lambda_i}:i\geq 1\}\geq 2|a_{j_0}| e^{\alpha \lambda_{j_0}}\ \forall \sigma\in(\alpha,\sigma_1)
\end{equation*}
or (condition b))
there exists
$\sigma_2\in(\beta-\varepsilon,\beta)$ such that
\begin{equation*}\label{lki}
\sum_{i\geq1}|a_i|e^{\sigma\lambda_i}>2\sup\{|a_i|e^{\beta\lambda_i}:i\geq 1\}\geq 2|a_{j_0}| e^{\beta \lambda_{j_0}}\ \forall \sigma\in(\sigma_2,\beta).
\end{equation*}
Therefore, by continuity, there exists $0<\tau<\min\{\sigma_1-\alpha,\beta-\sigma_2\}$, sufficiently small, such that
\begin{equation*}\label{lki1200}
\sum_{i\geq1}|a_i|e^{\sigma\lambda_i}>2|a_{j_0}| e^{(\alpha+\tau) \lambda_{j_0}}\ \forall \sigma\in(\alpha,\sigma_1)
\end{equation*}
or
\begin{equation*}\label{lki00}
\sum_{i\geq1}|a_i|e^{\sigma\lambda_i}>2|a_{j_0}| e^{(\beta-\tau) \lambda_{j_0}}\ \forall \sigma\in(\sigma_2,\beta).
\end{equation*}
In particular, we get
\begin{equation}\label{lki1200}
\sum_{i\geq1}|a_i|e^{(\alpha+\tau)\lambda_i}>2|a_{j_0}| e^{(\alpha+\tau) \lambda_{j_0}}
\end{equation}
or
\begin{equation}\label{lki00}
\sum_{i\geq1}|a_i|e^{(\beta-\tau)\lambda_i}>2|a_{j_0}| e^{(\beta-\tau) \lambda_{j_0}}.
\end{equation}
Since (\ref{lki1200}) and (\ref{lki00}) imply $g(\alpha+\tau)>0$ and $g(\beta-\tau)>0$, respectively, we get a contradiction
with (\ref{hut0}).
Now the result follows.
\end{proof}

We next focus our attention on the real solutions of
the equations
\begin{equation}\label{4.1}
\left\vert a_{j}\right\vert e^{\lambda_{j}\sigma }=\sum_{i\geq 1\text{,
}i\neq j}\left\vert a_{i}\right\vert e^{\lambda_{i}\sigma },\ j=1,2,\ldots,k,\ldots.
\end{equation}

\begin{lemma}\label{lem8}
Let $f(s)$ be an almost periodic function in a vertical strip $U=\{s\in\mathbb{C}:\alpha<\operatorname{Re}s<\beta\}$ whose Fourier exponents $\{\lambda_1,\lambda_2,\ldots,\lambda_k,\ldots\}$, with $k>2$, are $\mathbb{Q}$-linearly independent. Then each equation (\ref{4.1})
has at most $2$ real solutions in $(\alpha,\beta)$.
\end{lemma}
\begin{proof}
Let $f(s)$ be an almost periodic function in $U=\{s\in\mathbb{C}:\alpha<\operatorname{Re}s<\beta\}$ whose Dirichlet series is given by $\sum_{n\geq 1} a_ne^{\lambda_n s}$
with $\{\lambda_1,\lambda_2,\ldots,\lambda_k,\ldots\}$ $\mathbb{Q}$-linearly independent and $k>2$.
Fixed $j=1,2,...,k,\ldots$, 
we define the real function
\[
f_{j}(\sigma ):=\sum_{i\geq 1,i\neq j}\left\vert a_{i}\right\vert e^{\lambda_{i}\sigma
}-\left\vert a_{j}\right\vert e^{\lambda_{j}\sigma},\ \sigma\in(\alpha,\beta).
\]%
Also, by dividing  by
$|a_j|e^{\lambda_j\sigma}$, consider
 \begin{equation*}\label{equalityj}
 B_j(\sigma):=\frac{f_{j}(\sigma )}{
 |a_j|e^{\lambda_j\sigma}}=\sum_{i\geq 1,i\neq j}\Frac{|a_i|}{|a_j|}e^{(\lambda_i-\lambda_j)\sigma}-1,\ \sigma\in(\alpha,\beta).
 \end{equation*}
It is worth noting that, since $f(s)$ is analytic in $U$, then it is uniformly continuous in any open interval interior to $U$ together with all its derivatives \cite[p. 142]{Besi}.
Thus it is easy to check that 
$B_j^{''}(\sigma)\geq 0$ for all $\sigma\in (\alpha,\beta)$, i.e. $B_j(\sigma)$ is convex in $(\alpha,\beta)$.
Consequently, equation $B(\sigma)=0$, $\sigma\in (\alpha,\beta)$,  has at
most two solutions. Thus the result holds.
\end{proof}

At this point, we prove the following important result which generalizes \cite[Theorem 7]{Gaps}.
\begin{theorem}\label{theorem}
The set of the real projections of the zeros of an almost periodic function
 in an open vertical strip $U$ whose Fourier exponents $\{\lambda_1,\lambda_2,\ldots,\lambda_k,\ldots\}$,
 with $k>2$, are $\mathbb{Q}$-linearly independent has no isolated point in $U$.
\end{theorem}
\begin{proof}
If the real projection of a zero $s_0\in U$ of $f(s)$, say $\sigma_0$, were an isolated point of
the set $\{\operatorname{Re}s : f(s) = 0,\ s\in U\}$, necessarily $\sigma_0$ would be a boundary point
of the set $R_f=\overline{\left\{ \operatorname{Re}s:f(s)=0,\ s\in U\right\}}\cap (\alpha,\beta)$, with $\sigma_0\in(\alpha,\beta)$.
By Proposition \ref{lem2},
it satisfies all
the inequalities (\ref{PGp}), that is,
\begin{equation*}
\left\vert a_{j}\right\vert e^{\sigma_0 \lambda_{j}} \leq \sum_{i\geq1\text{, }%
i\neq j}\left\vert a_{i}\right\vert e^{\sigma_0 \lambda_{i}} \text{, }%
\left( j=1,2,\ldots,k,\ldots\right),
\end{equation*}%
 and only one of them is an
equality, say
\begin{equation}\label{4.200}
\left\vert a_{k}\right\vert e^{\lambda_{k}\sigma }=\sum_{i\geq 1\text{,
}i\neq k}\left\vert a_{i}\right\vert e^{\lambda_{i}\sigma }.
\end{equation}
Now,
from Lemma \ref{lem8}, we have that equation (\ref{4.200}), which is satisfied by $\sigma _{0}\in(\alpha,\beta)$, has $1$ or $2$ solutions in $(\alpha,\beta)$. This means that the equation $B_k(\sigma)=0$ has $1$ or $2$ solutions in $(\alpha,\beta)$, where
 \begin{equation*}\label{equalityj}
 B_k(\sigma):=\sum_{i\geq 1,i\neq k}\Frac{|a_i|}{|a_k|}e^{(\lambda_i-\lambda_k)\sigma}-1,\ \sigma\in(\alpha,\beta).
 \end{equation*}
Thus, since $B_k(\sigma)$ is continuous and convex in $(\alpha,\beta)$ (see also the proof of Lemma \ref{lem8}), we can assure the existence of
some $\varepsilon>0$ such that any
$\sigma $ in the interval $\left( \sigma _0-\varepsilon,\sigma
_{0}\right)\subset (\alpha,\beta) $ or $\left(\sigma
_{0}, \sigma _0+\varepsilon\right)\subset (\alpha,\beta) $ satisfies $B_k(\sigma)\geq 0$, which implies that $\sigma$ satisfies inequalities (\ref{PGp}). Then, by
Theorem \ref{point}, the interval $\left( \sigma _0-\varepsilon,\sigma
_{0}\right) $ or $\left(\sigma
_{0}, \sigma _0+\varepsilon\right) $ is in $R_f$, which is a contradiction because $\sigma_0$ is an isolated point in $R_f$.
\end{proof}

Under the conditions of the previous result, it is now clear that the set $R_f$ is the union of a denumerable amount of disjoint nondegenerate intervals.
In this respect, the following result concernes the gaps of the set $R_f$.

\begin{corollary}
Let $f(s)$ be an almost periodic function in a vertical strip $U=\{s\in\mathbb{C}:\alpha<\operatorname{Re}s<\beta\}$
whose Fourier exponents $\{\lambda_1,\lambda_2,\ldots,\lambda_k,\ldots\}$, with $k>2$, are $\mathbb{Q}$-linearly independent.
Then the gaps of $R_f$ are produced by those equations (\ref{4.1}) having two real solutions in $(\alpha,\beta)$.
\end{corollary}
\begin{proof}
Let $\sigma_0\in (\alpha,\beta)$ be a boundary point of $R_f$. Then, by Proposition \ref{lem2}, $\sigma_0$ satisfies only one of equalities (\ref{4.1}), say
 \begin{equation}\label{4.2002}
\left\vert a_{k}\right\vert e^{\lambda_{k}\sigma }=\sum_{i\geq 1\text{,
}i\neq k}\left\vert a_{i}\right\vert e^{\lambda_{i}\sigma }.
\end{equation}
If we suppose that equation (\ref{4.2002}) has only the solution $\sigma_0$ in $(\alpha,\beta)$, it follows from theorems \ref{point} and \ref{theorem} that
$$\left\vert a_{k}\right\vert e^{\lambda_{k}\sigma }<\sum_{i\geq 1\text{,
}i\neq k}\left\vert a_{i}\right\vert e^{\lambda_{i}\sigma }\ \forall \sigma\in(\alpha,\beta)\setminus\{\sigma_0\}.
$$
Therefore, in virtue from continuity and Theorem \ref{point}, we can assure the existence of some $\varepsilon>0$ such that the interval $\left( \sigma _0-\varepsilon, \sigma _0+\varepsilon\right) $ is in $R_f$, which is a contradiction because $\sigma_0$ is a boundary point in $R_f$. Consequently, by Lemma \ref{lem8}, equation (\ref{4.2002}) has two solutions in $(\alpha,\beta)$, which means that the gaps of $R_f$ are only produced by those equations (\ref{4.1}) having two real solutions in $(\alpha,\beta)$. 
\end{proof}

\begin{remark}
Every result in this section has been formulated for the case where the set of Fourier exponents has at least three elements and all of them are $\mathbb{Q}$-linearly independent. However, these results are also certain for other cases such as that where one of the Fourier exponents is $\lambda_1=0$ and the set of the remaining Fourier exponents has at least two elements and all of them are $\mathbb{Q}$-linearly independent. Indeed, if $f(s)$ is an almost periodic function whose Fourier exponents $\{0,\lambda_2,\lambda_3,\ldots\}$ satisfy these new conditions, then the function $g(s)=f(s)e^{\mu s}$, where $\mu$ is chosen so that it is not in the $\mathbb{Q}$-vector space generated by $\{\lambda_2,\lambda_3,\ldots\}$, has the same set of zeros as that of $f(s)$ and its Fourier exponents satisfy the conditions of the results of this section.

Finally, note that for the case where the set of Fourier exponents of an almost periodic function $f$ has only two elements (and they are $\mathbb{Q}$-linearly independent) it is clear that the zeros of $f(s)$ are located on a vertical line, and consequently Theorem \ref{theorem} is not satisfied in this case.
\end{remark}

\bigskip



\begin{thebibliography}{9}

\bibitem{Ash} R.B. Ash, W.P. Novinger, \textit{Complex variables}, Academic Press, New York,
2004. 

\bibitem{Avellar} C.E. Avellar, J.K. Hale,  On the zeros of exponential
polynomials, {\em J. Math. Anal. Appl.},
\textbf{73} (1980), 434--452.

\bibitem{Besi} A.S. Besicovitch, {\em Almost periodic functions}, Dover, New York, 1954.

\bibitem{BohrAnalytic} H. Bohr, Zur Theorie der fastperiodischen Funktionen. (German) III. Dirichletentwicklung analytischer Funktionen, \textit{Acta Math.} \textbf{47} (3) (1926), 237-281.

\bibitem{Bohr} H. Bohr, \textit{Almost periodic functions}, Chelsea, New York,
1947.

\bibitem{Borwein} P. Borwein, G. Fee, R. Ferguson, A. van der Waall, Zeros of partial sums of the Riemann zeta function, \textit{Exp. Math.}, \textbf{16} (1) (2007), 21--39.


\bibitem{RACSAM}
E. Dubon, G. Mora, J.M. Sepulcre, J.I. Ubeda, T. Vidal, {\em
A note on the real projection of the zeros of partial sums of the Riemann zeta function},
Rev. R. Acad. Cienc. Exactas F\'{i}s. Nat. Ser. A Math. RACSAM, \textbf{108} (2) (2014), 317--333.


\bibitem{Corduneanu1} C. Corduneanu, \textit{Almost periodic functions}, Interscience publishers, New York, London, Sydney, Toronto, 1968.

\bibitem{Corduneanu} C. Corduneanu, \textit{Almost periodic oscillations and waves}, Springer, New York, 2009.


\bibitem{Farag}
H.M. Farag, Dirichlet truncations of the Riemann zeta-function in
the critical strip possess zeros near every vertical line, \textit{Int. J.
Number Theory}, \textbf{4} (4), (2008) 653-662.




\bibitem{Jessen} B. Jessen, Some aspects of the theory of almost periodic functions, in \textit{Proc. Internat. Congress Mathematicians Amsterdam}, 1954, Vol. 1, North-Holland, 1954, pp. 304--351.

\bibitem{Gaps}
G. Mora, J.M. Sepulcre, T. Vidal, On the existence of
exponential polynomials with prefixed gaps, \textit{Bull. Lond. Math.
Soc.}, \textbf{45} (6), (2013) 1148-1162.

\bibitem{Moreno} C.J. Moreno, The zeros of exponential polynomials (I), \textit{Compos.
Math.}, \textbf{26} (1) (1973), 69--78.

\bibitem{PT} D.J. Platt, T.S. Trudgian, Zeroes of partial sums of the zeta-function, \textit{LMS J. Comput. Math.}, \textbf{19} (1), (2016) 37-41.


\bibitem{Sep} J.M. Sepulcre, On the result of invariance of the closure set of the real projections of the zeros of an important class of exponential polynomials, \textit{J. Funct. Spaces}, vol. 2016, Article ID 3605690, 2016.

\bibitem{JMT} J.M. Sepulcre, T. Vidal, A new approach to obtain points of
the closure of the real parts of the zeros of the partial sums
$1+2^z+\ldots+n^z,\ n\geq 2$, \textit{Kybernetes}, \textbf{41} (2012), 96--107.



\bibitem{SV3} J.M. Sepulcre, T. Vidal, On the non-isolation of the real projections of the zeros of exponential polynomials, \textit{J. Math. Anal. Appl.}, \textbf{437} (1), (2016) 513-525.


\bibitem{SV} J.M. Sepulcre, T. Vidal, Almost periodic functions in terms of Bohr's equivalence relation, Ramanujan J., \textbf{46} (1) (2018), 245--267. Corrigendum sent to the journal. See also {\em arXiv: 1801.08035 [math.CV]}.


\bibitem{SV1} J.M. Sepulcre, T. Vidal, A generalization of Bohr's equivalence theorem, {\em arXiv:1711.04112v2  [math.CV]}, (2018).


\bibitem{SV2} J.M. Sepulcre, T. Vidal, Sets of values of almost periodic functions, {\em arxiv:1801.08864 [math.CV]}, (2018).


\bibitem{Spira2} R. Spira, Zeros of sections of the Zeta function II., Math. Comp., \textbf{22} (101) (1968), 163-173.


\end{thebibliography}
\end{document}